\documentclass[12pt]{article}
\usepackage{}
\usepackage{amssymb}
\usepackage{amsfonts}
\usepackage{amsmath}
\usepackage{amscd}
\usepackage{amsthm}
\usepackage{latexsym}
\usepackage[all,arc,2cell]{xy}
\usepackage{CJK}
\numberwithin{equation}{section}

\newcommand {\Dim}{\textrm{dim}}
\newcommand {\irr}{\textrm{Irr}}
\newcommand {\Img}{\textrm{Im}}
\newcommand {\codim}{\textrm{codim}}
\newcommand {\Rep}{\textrm{Rep}}
\newcommand {\Hom}{\textrm{Hom}}
\newcommand {\Ext}{\textrm{Ext}}
\newcommand {\Mod}{\textrm{mod}}

\newcommand{\frg}{\mathfrak{g}}

\newtheorem{theo}{Theorem}[section]
\newtheorem{lem}{Lemma}[section]
\newtheorem{cor}{Corollary}[section]
\newtheorem{remark}{Remark}[section]
\date{}

\begin{document}
\begin{CJK*}{GBK}{song}
\title{\bf The matrix between PBW basis and semicanonical basis of $U^+(sl_n(\mathbb{C}))$$^\star$}
\author{{Hongbo Yin,   Shunhua Zhang}\\
        {\small School of Mathematics, Shandong University,
        Jinan, 250100,P.R.China}}
\maketitle

\begin{abstract} In this paper, we prove that the matrix between a special PBW basis
and the semicanonical basis of $U^+(sl_n(\mathbb{C}))$ is upper
triangular unipotent under any order which is compatible with the
partial order $\leq_{deg}$.
\end{abstract}

\vskip0.1in

{\bf Key words and phrases:}\ Representations of quivers,
degeneration, preprojective algebras, semicanonical basis.

\vskip0.1in

{\bf MSC(2000):}  16G20, 17B37

\footnote{ $^\star$Supported by the NSF of China (Grant No.
11171183)}

\footnote{  Email addresses: \ yinhongbo0218@126.com(H.Yin),
 \ shzhang@sdu.edu.cn(S.Zhang)}

\section{Introduction}

Assume $Q$ is a simply-laced Dynkin graph, and $\frg$ is the corresponding simple Lie algebra.
Let $U_q^+$ be the plus part of the quantized enveloping algebra of $\frg$.
It has the PBW type basis which is important in the study of $U_q^+$.
But the the PBW type basis is dependent of the orientation we choose for $Q$.
In \cite{L1}, Lusztig defined the canonical basis of $U_q^+$.
Later, the canonical basis was extended to general case by Lusztig in \cite{L3}
and independently by Kasiwara in \cite{K} called crystal basis.
The canonical basis is parameterized by the irreducible
components of the nilpotent variety of the preprojective
algebra of $Q$ \cite{L2,KS} and is independent of the orientation
we choose for $Q$ and has a lot of good properties. When $q$ limits to $1$,
we get the canonical basis of $U^+$, the plus part of enveloping algebra of $\frg$.
In \cite{L4}, Lusztig defined another basis of $U^+$ called semicanonical basis.
The semicanonical basis doesn't coincide with the canonical basis in general,
but they share some common properties. For example, they are compatible with
various filtrations of $U^+$, compatible with the canonical antiautomorphism of
$U^+$ and projects to a basis of irreducible modules.
It is known that the matrix between the PBW basis and the
canonical basis is upper triangular unipotent (implicit in the proof of the existence of
canonical bases of Lusztig in \cite{L1}).
So we may ask wether the assertion is true for semicanonical basis?
In this paper we show the answer is positive for a special PBW basis of type $A_n$.
That is

\begin{theo} Let $\frg$ be $sl_n(\mathbb{C})$ and $U^+$ be
the universal enveloping algebra of the positive part of $\frg$.
We choose the linear orientation $\overrightarrow{A_n}$ for the
Dynkin graph $A_n$ of $\frg$. Let $\{f_M\}_{M\in \Rep(\overrightarrow{A_n}, V)}$
and $\{P_M\}_{M\in \Rep(\overrightarrow{A_n}, V)}$ be the semicanonical basis
and the PBW basis of $U_V^+$. Then, $f_M=\sum_{M\leq_{deg}N}a_NP_N$ and $a_M=1$.
\end{theo}

The same question was investigated for simply-laced Dynkin type in
\cite{BK}. Here we want to use the representations of quivers to
understand it, and our method is very different and more direct.

This paper is organized as follows: in section 2,
we recall the definition and some facts about the semicaonical basis;
in section 3, we make some preparations and then prove our main result.

{\bf Acknowledgements}\ \  The authors would like to thank professor
Jan.Schr\"{o}er for answering a lots of questions and pointing the
paper of P.Baumann and J.Kamniter. The authors also would like to
thank professor K.Bongartz to provide the proof of lemma 3.3 and
lemma 3.4 below.

\section{Preliminary}

\subsection{Enveloping algebra}
Let $Q$ be a simply-laced Dynkin graph and $Q_0$ be the set of vertices. Let $a_{ij}$ be the minus number
of the edges connecting the vertices $i$ and $j$ if $i\neq j$ and $a_{ii}=2$. Then we get a Cartan datum.
The algebra $U^+$ is the associative algebra over $\mathbb{C}$ generated by $\{e_i\}_{i\in Q_0}$ with Serre relations

\[\sum_{p,q\in \mathbb{N};p+q=-a_{ij}+1}(-1)^p\frac{e_i^p}{p!}e_j\frac{e_j^q}{q!}=0\]
for $i\neq j$.

Let $V=\bigoplus_{i\in Q_0}V_i$ be a $Q_0$-graded vector space. We write $|V|=\sum_i\Dim V_i i\in \mathbb{N}[Q_0]$.
Let $U_V^+$ be the subspace of $U^+$ generated by the monomials $e_{i_1}e_{i_2}\cdots e_{i_n}$ for various
sequence $i_1, i_2, \cdots, i_n$ in which $i$ appears $\Dim V_i$ times for any $i\in Q_0$. This gives a grade of $U^+$.

\subsection{Preprojective algebra}
We choose an orientation for $Q$, then $Q=(Q_0, Q_1)$ became a Dynkin quiver of ADE type, where $Q_0$ is
the set of vertices and $Q_1$ is the set of arrows. Given an arrow $\alpha$, we denote by $s(\alpha)$ the
start point of $\alpha$ and by $t(\alpha)$ the end point of $\alpha$. Let $\overline{Q}$ be the double quiver of $Q$,
which is obtained from $Q$ by adding an arrow $\alpha^\ast: j\rightarrow i$ whenever there is an arrow $\alpha: i\rightarrow j$ in $Q$.
Let $Q_1^*=\{\alpha^*| \alpha\in Q_1\}$ and $\overline{Q}_1=Q_1\cup Q_1^*$. The preprojective algebra associated with $Q$ is defined as

\[\Lambda=\Lambda_Q=K\overline{Q}/(c_i)_{i\in Q_0}\]
where $c_i$ is the relation

\[c_i=\sum_{\alpha\in \overline{Q}_1, s(\alpha)=i}(\alpha^\ast\alpha)\]
and $K\overline{Q}$ is the path algebra of $\overline{Q}$. Note that the preprojective algebra is independent of the orientation of $Q$.

\subsection{Representation variety}
Given a quiver $Q$, let $V=\bigoplus_{i\in Q_0}V_i$ be a $Q_0$-graded vector space.
Let $\Rep(Q, V)=\prod_{h\in Q_1}\Hom_k(V_{s(h)}, V_{t(h)})$ be the representation variety of the quiver $Q$.
There is an action of algebraic group $G_V=\prod_{i\in Q_0}GL(V_i)$ on $\Rep(Q, V)$ which is defined by $g\cdot x=(g_{t(h)}x_hg_{s(h)}^{-1})$.
Let $M\in \Rep(Q,V)$, we denote $\mathcal {O}_M$ the orbit of $M$ under the action of $G_V$.
We denote by $\Lambda_V$ the nilpotent variety of the preprojective algebra $\Lambda$,
that is the elements $(x_h)$ of $\Rep(\overline{Q}, V)$ satisfying: (1) for all $i\in Q_0$,
$\sum_{h\in Q_1, s(h)=i}x_{h^*}x_{h}=\sum_{h\in Q_1, t(h)=i}x_{h}x_{h^*}$; (2) there is
an integer $N$ such that for any sequence $h_1, h_2, \cdots, h_n$ with $n\geq N$ in $\overline{Q}_1$,
the $x_{h_n}\cdots x_{h_2}x_{h_1}$ is zero. There is a functor $\pi_Q$ from $\Lambda_V$ to $E_Q$ which
sends $x_h$ to $x_h$ when $h\in Q_1$ and to zero when $h\in Q_1^*$.

For Dynkin case, it was proved that $\Lambda_V$ is just all the $\Lambda$-modules with
underlying vector space $V$ in \cite{L3}. Let $\irr~\Lambda_V$ be the set of the irreducible components
of $\Lambda_V$. It is known that every element of $\irr~\Lambda_V$ is of the form $\overline{\pi_Q^{-1}(\mathcal{O}_M)}$
where $M$ is the element of $\Rep(Q, V)$ and $\overline{\pi^{-1}(\mathcal{O}_M)}$ means the
closure of $\pi_Q^{-1}(\mathcal{O}_M)$ in the variety $\Lambda_V$ \cite{L3}. So we denote the
irreducible component of $\Lambda_V$ by $Z_M$ where $M$ means an isomorphism class in $\Rep(Q, V)$.
The affine space $\Rep(Q, V)$ is a subset of $\Lambda_V$ and it is an irreducible component of $\Lambda_V$.

Let
\[\Lambda_{V,i,p}=\{(x_h)\in \Lambda_V\ |\ \codim_{V_i}(\sum_{h\in\overline{Q_1}, t(h)=i}\Img~(x_h:V_{s(h)}\rightarrow V_i))=p\}.\]
For $i\in Q_0$ and $Z\in \irr~\Lambda_V$, there is a unique $p$ such that $\Lambda_{V,i,p}\cap Z$ is
open dense in $Z$. Set $t_i(Z)=p$. This is a function from $\irr~\Lambda_V$ to $\mathbb{N}$ defined by Lusztig
in \cite{L4}. Similarly, for $M\in \Rep(Q,V)$, we define $t_i(M)=\codim_{V_i}(\sum_{h\in Q_1, t(h)=i}\Img(x_h:V_{s(h)}\rightarrow V_i))$.

Let $x\in \Lambda_{V,i,p}$. Then there is a unique subrepresentation $x'$ of $x$ such that $x'\in \Lambda_{V',i,0}$, where $|V|=|V'|+pi$.
This was proved in \cite[12.5]{L4}.

\subsection{Semicanonical basis and PBW basis}
A subset of an affine variety  is called constructible if it is a finite union of locally closed subsets. A function from an
affine variety to a vector space is called constructible if the image set is finite and the converse image of a point is a constructible subset.

Let $\widetilde{\mathcal{M}}_V$ be the vector space of constructible functions from $\Lambda_V$ to $\mathbb{C}$ which is
constant on the same orbit and let $\widetilde{\mathcal{M}}=\bigoplus_V \widetilde{\mathcal{M}}_V$. Then $\widetilde{\mathcal{M}}$
becomes an associative algebra under the product $\ast$ defined below.
Let $|V|=|V'|+|V''|$, $f\in \widetilde{\mathcal{M}}_{V'}, g\in \widetilde{\mathcal{M}}_{V''}$ and $x\in \Lambda_{V}$, define

\[(f\ast g)(x)=\int_{x''}f(x/x'')g(x'')\]
where $x''$ is the subrepresentation of $x$ with underlying vector space $V''$.

We denote by $S_i$ the simple representation corresponding to vertices $i$ and $1_i$ the function
which maps $\mathcal {O}_{S_i}$ to 1. Let $\mathcal{M}$ be the subalgebra of $\widetilde{\mathcal{M}}$
generated by $\{1_i\}$, then there is an algebra isomorphism $\Phi$ between $\mathcal{M}$ and $U^+$ which sends $1_i$ to $e_i$ \cite[12.13]{L3}.

Let $f\in \widetilde{\mathcal{M}}$. For each $Z_M\in \irr~\Lambda_V$, there is a unique $c\in \mathbb{Q}$
such that $f^{-1}(c)\cap Z_M$ contains an open dense subset of $Z_M$. We denote the function sending $f$ to $c$ by $\rho_{Z_M}$.

\begin{theo}[Lusztig] Let $Z\in \rm Irr~\Lambda_V$. There exists $f\in \mathcal{M}_V$ such that $\rho_Z(f)=1$
and $\rho_{Z'}(f)=0$ for any $Z'\in\rm Irr~\Lambda_V-\{Z\}$. This is a basis of $\mathcal{M}_V$.
\end{theo}

The basis $\{f_Z\}$ is called the semicanonical basis of $\mathcal{M}_V$ and its image under $\Phi$ is called
the semicanonical basis of $U^+_V$ \cite{L4}. We will identify them from now on and denote them both by $B_V$.
Then $B=\cup_VB_V$ is called the semicanonical basis of $U^+$. For the Dynkin case, we also write $\{f_M\}$
for $\{f_{Z_M}\}$. Although the semicanonical basis is independent of the choice of orientation of the quiver,
the parametrization of the semicanonicl basis depents on, thus our denotation can't lead to misunderstanding.

For a Dynkin quiver $Q$, Ringel showed that the algebra $U^+$ has a basis parameterized by the isomorphism classes
of representations of the quiver, which is called PBW basis and we denote it by $\{P_M\}$,
see \cite{S}. In \cite[5.9]{GLS}, Gei\ss, Leclerc and Schr\"{o}er describe them in $\mathcal{M}$:

\begin{lem} Let $\mathcal{O}$ be a $G_V$-orbit in $\rm Rep(Q,V)$. There exists a unique $\mathcal{K}_{\mathcal{O}}\in \mathcal{M}$
whose restriction to $\rm Rep(Q,V)$ is the characteristic function of $\mathcal{O}$.
The collection of all $\mathcal{K}_{\mathcal{O}}\in \mathcal{M}$ where $\mathcal{O}$
runs through all $G_V$-orbit in $\rm Rep(Q,V)$ is equal to $\Phi^{-1}(\{P_M\})\cap \mathcal{M}$.
\end{lem}

In other words, we can identify the elements of PBW basis with the characteristic functions of the orbits in $\Rep(Q,V)$.

\subsection{Some geometry of representations of finite dimensional algebras}

This subsection is just a copy of the material in \cite{B1}.

Let $A$ be a finite dimensional associated $k$-algebras with basis $a_1=1, \cdots, a_n$.
We have the corresponding structure constants $a_{ijk}$ defined by $a_ia_j=\sum a_{ijk}a_k$.
The affine variety $Mod^d_A$ of $d$-dimensional $A$-modules is given by the $n$-tuples $m=(m_1,\cdots,m_n)$ of $d\times d$-matrices
such that $m_i$ denotes the action of $a_i$ and they satisfy $m_im_j=\sum a_{ijk}m_k$ for all $i$ and $j$.
Let $Gl_d(k)$ be the $d^2$ dimensional general linear group which acts on $Mod^d_A$ by conjugation. The orbits correspond
to the isomorphism classes of $d$-dimensional modules.

Let $r$, $t$ and $s=r+t$ be three natural numbers. Let $\mathcal{U}$ be an irreducible subvariety of $Mod^r_A$ and $\mathcal{V}$
be an irreducible subvariety of $Mod^t_A$. The space of 1-cocycles is the tuples $z=(0,z_2,\cdots,z_n)$ in $k^{r\times t}$
such that $\sum a_{ijk}z_k=u_iz_j+z_iv_j$ holds for all $i$ and $j$. Thus $Z(v,u)$ is the set of solutions of a system of
homogeneous linear equations whose coefficients depend polynomially on the entries of $u$ and $v$.
Therefore, the map $(v,u)\rightarrow \Dim Z(v,u)$ is upper semi-continuous. Let $B(v,u)$
be the subspace of coboundaries which is the image of the linear map from $k^{r\times t}$ to $Z(v,u)$
sending $h$ to the tuple with the $i$th entry $hv_i-u_ih$. Because $\Ext_A^1(V,U)=Z(V,U)/B(V,U)$, we have

\begin{eqnarray}\Dim\ Z(V,U)&=&\Dim~\Ext_A^1(V,U)+\Dim\ B(V,U)\\
                            &=&\Dim~\Ext_A^1(V,U)-\Dim~\Hom_A(V,U)+rt.
\end{eqnarray}
Thus $\Dim Z(V,U)$ is constant if and only if $\Dim~\Ext_A^1(V,U)-\Dim~\Hom_A(V,U)$ is. In this case, we have a vector bundle
\[ p:\mathcal{L}\rightarrow \mathcal{U}\times\mathcal{V}\]
with irreducible total space

\[ \mathcal{L}=\left\{\left(\begin{array}{cc}u&z\\0&v\end{array}\right)\bigg|u\in \mathcal{U}, v\in \mathcal{V}, z\in Z(v,u)\right\}.\]
The image $\mathcal{L}'$ of the conjugation

\[Gl_s\times\mathcal{L}\rightarrow Mod^s_A \]
is the irreducible constructible set of all extensions of some $V$ in $\mathcal{V}$ by some $U$ in $\mathcal{U}$.

\section{Main result}

In this section we set $\frg=sl_{n+1}(\mathbb{C})$, then the Dynkin graph of $\frg$ is of type A. Let $\overrightarrow{A_n}$ be the quiver as follows:

$$\xymatrix{1\ar[r]&2\ar[r]&\cdots\ar[r]&n-1\ar[r]&n}.$$

Let $\Lambda$ be a finite dimensional algebra and $M,N\in \Mod\ \Lambda$. We say $M$ degenerate to $N$ and
denote by $M\leq_{deg}N$, if $\mathcal{O}_N\subset\overline{\mathcal{O}}_M$. This is a partial order of $\Rep(V,Q)$,
see \cite{B1}. There is also another partial order $\leq$ define as follows. $M\leq N$
if and only if $\Dim\ \Hom(M,L)\leq \Dim\ \Hom(N,L)$ for any $\Lambda$-module $L$. In general,
we can deduce $M\leq N$ from $M\leq_{deg}N$, and the converse is not true. But for direct algebras,
especially path algebras, $\leq$ and $\leq_{deg}$ are equivalent, see \cite{B2}.

Let $A=k\overrightarrow{A_n}$ be path algebra of $\overrightarrow{A_n}$ over an algebraic closed field $k$,
then the elements of $\Rep(\overrightarrow{A_n},V)$ can be viewed as $A$-modules and we will not distinguish them.

\begin{lem} $t_n(Z_M)=t_n(M)$

\end{lem}

\begin{proof} Since $n$ is the sink point of $\overrightarrow{A_n}$, it is the source point of $\overrightarrow{A_n}^{op}$.
For any $M'\in \pi^{-1}(\mathcal{O}_M) $,

\[t_i(M')=\codim(\Img(x_h:V_{n-1}\rightarrow V_n))=t_i(M)\]

So $\Lambda_{V,n,t_n(M)}\cap Z_M$ contains $\pi^{-1}(\mathcal{O}_M)$ and must contain an open dense subset. So $t_n(Z_m)=t_n(M)$.

\end{proof}

\begin{lem}
           If $t_{i+1}(M)=0$, then $t_i(Z_M)=t_i(M)$
\end{lem}

\begin{proof} The local part of $M$ looks as follows

\[\UseAllTwocells
\xymatrix@=20mm{V_{i-1}\rtwocell<2>^{x_{h_1}}_{x_{h_1^*}}{'}&V_i\rtwocell<2>^{x_{h_2}}_{x_{h_2^*}}{'}&V_{i+1}}.\]
Let $M'\in \pi^{-1}(\mathcal{O}_M) $.
By the relation of preprojective algebras, we have $x_{h_1}x_{h_1^*}+ x_{h_2^*}x_{h_2}=0$. Since $x_{h_2}$ is surjective,
we deduce $\Img x_{h_2^*}\subset \Img x_{h_1}$. By the definition, $t_i(M')=\codim(\Img x_{h_1})=t_i(M)$.
So $\Lambda_{V,i,t_i(M)}\cap Z_M$ contains $\pi^{-1}(\mathcal{O}_M)$ and must contain an open dense subset. So $t_i(Z_M)=t_i(M)$.
\end{proof}

The following two lemmas are proved by K.Bongartz and they hold for any Dynkin quivers not only our choice.

\begin{lem}
Let $M$ be an $A$-module and $S$ a simple $A$-module. Set
$S^m=\oplus_{i=1}^mS$. Let $\mathcal{T}$ be the set of all the
A-modules $N$ such that there is an exact sequence $0\rightarrow
M\rightarrow N\rightarrow S^m\rightarrow 0$. Then there is a module
$L$ such that all the other modules in $\mathcal{T}$ is a
degeneration of $L$ .
\end{lem}

\begin{proof}

Let $r$ be the total dimension of $M$ and let $t$ be the total dimension of $S^m$ and set $s=r+t$.
Now, we take for $\mathcal{U}$ the orbit of $M$ and for $\mathcal{V}$ the orbit of $S^m$ in section 2.5.
Then $\mathcal{L}'$ is irreducible because the space of cocycles has constant dimension for obvious reasons (formula (2.2)).
Since the irreducible set $\mathcal{L}'$ is a finite union of orbits - $A$ is representation-finite - there is a dense orbit. This is the wanted module $L$.

\end{proof}

We call the module $L$ in the lemma the generic extension of $S^m$ by $M$.

\begin{lem}Let $U\leq_{deg}W$. Set $S$ is a simple $A$-module and $S^m=\oplus_{i=1}^mS$
and there is an exact sequence $0\rightarrow W\rightarrow N\rightarrow S^m\rightarrow 0$,
then there is an exact sequence $0\rightarrow U\rightarrow M\rightarrow S_i^m\rightarrow 0$ such that $M\leq_{deg}N$.

\end{lem}

\begin{proof}
 We take for $\mathcal{U}$ the closure of the orbit of $U$ and for $\mathcal{V}$ the orbit
 of $S^m$ again, then $W\in \mathcal{U}$ and $N\in \mathcal{L}$. because $A$ is hereditary,
 $\Dim~\Ext_A^1(V,U)-\Dim~\Hom_A(V,U)$ is determined by the dimension vectors of $V$ and $U$,
 then the space of cocycles has constant dimension. Again there is a dense orbit in $\mathcal{L}'$
 because $\mathcal{L}'$ is a finite union of orbits. Call this module $M$. The orbit of $M$ is open
 and so is its non-empty intersection $F$ with $\mathcal{L}$. Since the bundle-projection $p$
 from $\mathcal{L}$ to $\mathcal{U}$ is an open map and since $\mathcal{U}$ is irreducible,
 the image $p(F)$ and the orbit of $U$ ( which is open in $U$ ) intersect. Then $M$ is an extension of $S^m$ by $U$. Thus $M$ is the wanted module.

\end{proof}

Let $Z_M\in\irr~\Lambda_V$ and $t_i(Z)=p$. Then
$Z=Z_M\cap\Lambda_{V,i,p}\in \irr~\Lambda_{V,i,p}$. Given $x\in Z$,
there is a unique submodule $x'$ of $x$ such that $x/x'\cong S_i^p$.
Denote the orbit of $x'$ by $O(x)$, then $Z'=\cup_{x\in Z_M}O(x)\in
\irr~\Lambda_{V',i,0}$ and $Z=\{x\in \Lambda_{V,i,p}|O(x)\in Z'\}$.
This gives a bijection between $\irr~\Lambda_{V,i,p}$ and
$\irr~\Lambda_{V',i,0}$, see \cite[2.3]{L4}. The closure of $Z'$ in
$\Lambda_{V'}$ is an irreducible component of $\Lambda_{V'}$ and can
be written as $Z_{M'}$. It is obvious that $M'$ is a submodule of
$M$, and $M/M'\cong S_i^p$.

\begin{lem}

Let $M$ be an $A$-module with $t_i(M)=m$, $t_{i+1}(M)=0$. Set $Z_{M'}\in \rm Irr~\Lambda_{V'}$
with $t_i(Z_{M'})=0$ corresponding to $Z_M$. Then $M$ is the generic extension of $S_i^m$ by $M'$.

\end{lem}

\begin{proof}

By Lemma 3.2, $t_i(Z_M)=t_i(M)=m$. By the description above, $M$ is an extension of $S_i^m$ by $M'$.
Now $t_{i+1}(M)=0$ means that $M$ is the generic extension and we will explain the reasons.

Let $M'$ be the representation with the local part of $i$ as follows:

\[\xymatrix@C=3cm@R=2cm{V_{i-1}\ar[r]^{x_1}&V_i\ar[r]^{x_2}&V_{i+1}},\]

Then any extension $N$ of $M'$ by $S_i^m$ is of the form

\[\xymatrix@C=3cm@R=2cm{V_{i-1}\ar[r]^{\scriptsize \left(\begin{array}{cc}x_1&0\end{array}\right)}&V_i\oplus V'_i\ar[r]^{\scriptsize
\left(\begin{array}{c}x_2\\x'_2\end{array}\right)}&V_{i+1}}.\]

If $t_{i+1}(N)=0$, then ${\rm rank}\scriptsize \left(\begin{array}{c}x_2\\x'_2\end{array}\right)=\Dim V_{i+1}$.
If $ \scriptsize \left(\begin{array}{c}x_2\\x''_2\end{array}\right)$ is another matrix with rank $\Dim V_{i+1}$,
then we can find an invertible matrix $\scriptsize \left(\begin{array}{cc}I&0\\h_2&h_3\end{array}\right)$ such
that $\scriptsize \left(\begin{array}{cc}I&0\\h_2&h_3\end{array}\right)\left(\begin{array}{c}x_2\\x''_2\end{array}\right)
=\left(\begin{array}{c}x_2\\x'_2\end{array}\right)$. So we have the communicative diagram

\[\xymatrix@C=3cm@R=2cm{V_{i-1}\ar[r]^{\scriptsize \left(\begin{array}{cc}x_1&0\end{array}\right)}\ar[d]^{id}&V_i
\oplus V'_i\ar[r]^{\scriptsize \left(\begin{array}{c}x_2\\x'_2\end{array}\right)}\ar[d]_{\scriptsize \left(\begin{array}{cc}
I&0\\h_2&h_3\end{array}\right)}&V_{i+1}\ar[d]^{id}\\
 V_{i-1}\ar[r]^{\left(\begin{array}{cc}x_1&0\end{array}\right)}&V_i\oplus V'_i\ar[r]^{\scriptsize \left(\begin{array}{c}x_2\\x''_2\end{array}\right)}&V_{i+1}
}.\]

Then it is easy to see that under the sense of isomorphism, there is only one extension $N$ of $S_i^m$ by $M'$ such that $t_{i+1}(N)=0$.

For any extension $L$ of $S_i^m$ by $M'$ with $t_{i+1}(L)>0$, we have $0=\Dim\ \Hom_A(M,S_i)<\Dim\ \Hom_A(L,S_i)$. So $L$ can't degenerate to $M$.
By the existence of generic extension of Lemma 3.3., we see that $M$ is the generic extension.

\end{proof}

\begin{lem} Let $M$ be an $A$-module with $t_i(M)=m>0$, $t_{i+1}(M)=0$. Set $Z_{M'}\in \rm Irr~\Lambda_{V'}$ with $t_i(Z_{M'})=0$
corresponding to $Z_M$. Then $\widetilde{f}_M=1_{S_i^m}*P_{M'}$ can be written as the linear combination of $\{P_N\}_{M\leq_{deg} N}$
and the coefficient of $P_M$ is 1.
\end{lem}

\begin{proof}

By the definition of $\widetilde{f}_M$, $\widetilde{f}_M(N)\neq0$ if and only if there is an exact
sequence $0\rightarrow M'\rightarrow N\rightarrow S_i^m\rightarrow 0$. By Lemma 3.5, $M$ is the
generic extension of $S_i^m$ by $M'$, So $M\leq_{deg}N$. Then by Lemma 2.1, $\widetilde{f}_M=1_{S_i^m}*P_{M'}$
can be written as the linear combination of $\{P_N\}_{M\leq_{deg}N}$. It is obvious that $\widetilde{f}_M(M)=1$, so the coefficient of $P_M$ is 1.

\end{proof}

\begin{lem}

Let $M$ be an $A$-module with $t_i(M)=m>0$, $t_{i+1}(M)=0$. Set
$Z_{M'}\in \rm Irr~\Lambda_{V'}$ with $t_i(Z_{M'})=0$ corresponding
to $Z_M$. Let $M'\leq_{deg}M''$, then $1_{S_i^m}*P_{M''}$ can be
written as the linear combination of $\{P_N\}_{M\leq_{deg}N}$.

\end{lem}

\begin{proof} $1_{S_i^m}*P_{M''}(N)\neq 0$ if and only if $N$ is an extension of $S_i^m$ by $M''$.
By Lemma 3.4, $M\leq_{deg} N $ because $M$ is the generic extension of  $S_i^m$ by $M'$. Then the lemma follows from lemma 2.1.

\end{proof}

Let $M$ be an $A$-module. Assume $0\subset M_k\subset M_{k-1}\subset\cdots\subset M_1\subset M_0=M$
is a submodule flag of $M$ such that $M/M_1\cong S_{i_1}^{m_1}, M_j/M_{j+1}\cong S_{i_j}^{m_j}$ and
for every $1\leq j\leq k-1$, $M_j$ is the generic extension of $S_{i_j}^{m_j}$ by $M_{i_{j+1}}$.
We call such a flag generic flag and denote it by $(M, S_{i_1}^{m_1},\cdots,S_{i_k}^{m_k})$.
If $M_k=S_{i_k}^{m_k}$, we say the flag is total.

\begin{lem} Let $M$ be an $A$-module and $(M, S_{i_1}^{m_1},\cdots,S_{i_k}^{m_k})$
be any generic flag of $M$. If $\rho_{Z_N}(1_{s_{i_1}^{m_1}}*\cdots*1_{s_{i_k}^{m_k}}*f_{M_k})\neq 0$, then $M\leq_{deg}N$.
\end{lem}

\begin{proof}

If $\dim M=1$ or $M$ is semisimple, it is easy to check. Now, assume the assertion is true
for all $L$ such that $\dim L<\dim M$ or $M\leq_{deg}L$ and assume that $M\nleqslant_{deg}N$

Set $t_{i_{k+1}}(M_k)=m_{k+1}>0, t_{i_{k+1}+1}(M_k)=0$ and $Z_{M_{k+1}}\in \irr~\Lambda_{V_{k+1}}$
with $t_{i_{k+1}}(Z_{M'})=0$ corresponding to $Z_{M_k}$, then $M_{k}$ is the generic extension
of $S_{i_{k+1}}^{m_{k+1}}$ by $M_{k+1}$ by lemma 3.5. By the construction of Lusztig \cite[the proof of Lemma 2.2]{L4}, we have

\[f_{M_k}=1_{S_{i_{k+1}}^{m_{k+1}}}*f_{M_{k+1}}-\sum_{Z_L\in \irr(\Lambda_{V}), t_i(Z_L)>m}\rho_{Z_L}(1_{S_{i_{k+1}}^{m_{k+1}}}*f_{M_{k+1}})f_{L}.\]
By the induction hypothesis,
\begin{eqnarray*} 1_{s_{i_1}^{m_1}}*\cdots*1_{s_{i_k}^{m_k}}*f_{M_k}&=&1_{s_{i_1}^{m_1}}*\cdots*1_{s_{i_k}^{m_k}}*(1_{s_{i_{k+1}}^{m_{k+1}}}*f_{M_{k+1}}\\
&&- \sum_{M_{k+1}\leq_{deg}X}a_Xf_X)\\
&=&1_{s_{i_1}^{m_1}}*\cdots*1_{s_{i_k}^{m_k}}*1_{s_{i_{k+1}}^{m_{k+1}}}*f_{M_{k+1}}\\
&&- 1_{s_{i_1}^{m_1}}*\cdots*1_{s_{i_k}^{m_k}}*\sum_{M_{k}\leq_{deg}X}a_Xf_X,
\end{eqnarray*}
where $a_X=\rho_{Z_X}(1_{S_{i_{k+1}}^{m_{k+1}}}*f_{M_{k+1}})$.

We set $(\overline{X}, S_{i_1}^{m_1},\cdots,S_{i_k}^{m_k})$ is the generic flag of $\overline{X}$
with $0\subset X\subset X_{k-1}\subset\cdots\subset X_1\subset \overline{X}$.

By lemma 3.4, we have $M\leq_{deg}\overline{X}$. Thus $\overline{X}\nleqslant_{deg}N$. By hypothesis,
$\rho_{Z_N}(1_{s_{i_1}^{m_1}}*\cdots*1_{s_{i_k}^{m_k}}*\sum_{M_{k+1}\leq_{deg}X}a_Xf_X)=0$.
So, we have $\rho_{Z_N}(1_{s_{i_1}^{m_1}}*\cdots*1_{s_{i_k}^{m_k}}*1_{s_{i_{k+1}}^{m_{k+1}}}*f_{M_{k+1}})\neq 0$.
Recursively, we get that $\rho_{Z_N}(1_{s_{i_1}^{m_1}}*\cdots*1_{s_{i_k}^{m_k}}*1_{s_{i_{n}}^{m_{n}}})\neq 0$,
where $(M, S_{i_1}^{m_1},\cdots,S_{i_n}^{m_n})$ is a total generic flag of $M$.

Then there is a dense subset $\mathcal{S}$ of $Z_N$ such that
the value of $1_{s_{i_1}^{m_1}}*\cdots*1_{s_{i_k}^{m_k}}*1_{s_{i_{n}}^{m_{n}}}$ on $\mathcal{S}$
is a nonzero constant and $\pi(\mathcal{S})$ is $\mathcal{O}_N$. For any $Y\in\mathcal{S}$,
there is a flag $0\subset Y_n\subset\cdots\subset Y$ of $Y$ such that $Y_j/Y_{j+1}=S_{i_j}^{m_j}$
for all $1\leq j\leq n-1$ and $Y_n=S_{i_n}^{m_n}$. we have $N$ also has such a flag because
any flag of $Y$ is a flag of $N$. Thus $M\leq_{deg} N$.
\end{proof}

\begin{remark} For our choice of $A$, any $A$-module $M$ has a total generic flag. In fact,
there always exists $i$ such that $t_i(M)=m>0$ and $t_{i+1}(M)=0$, then by lemma 3.5,
there is a submodule $M'$ such that $M$ is the generic extension of $S_i^m$ by $M'$.
Recursively, we can find a total generic flag of $M$.
\end{remark}

Now we can prove our main theorem.

\begin{theo} Let $\frg=sl_{n+1}(\mathbb{C})$ and $U^+$ be the positive part of the
universal enveloping algebra of $\frg$. We choose the linear orientation $\overrightarrow{A_n}$
for the Dynkin graph $A_n$ of $\frg$. Let $\{f_M\}_{M\in \Rep(\overrightarrow{A_n}, V)}$
and $\{P_M\}_{M\in \Rep(\overrightarrow{A_n}, V)}$ be the semicanonical basis and the PBW basis of $U_V^+$.
Then, $f_M=\sum_{M\leq_{deg}N}a_NP_N$ and $a_M=1$.
\end{theo}

\begin{proof}

We induct on the dimension of $V$. If $\Dim\ V=1$, the result is trivial.
We assume the result is right for $\Dim\ V\leq n$. For $\Dim\ V= n+1$, Let $M\in \Rep(V, Q)$,
we can always find a vertex $i$ such that $t_i(M)=m>0$ and $t_{i+1}(M)=0$ (just check from $n$).
If $M$ is semisimple, we can easily verify $f_M=P_M$. In fact, we assume $M=\oplus_{i\in Q_0}S_i^{r_i}$.
By Lusztig's construction \cite[the proof of Lemma 2.2]{L4}, $f_M=1_{S_n^{r_n}}*f_{M'}$ where $M'=M/S_n^{r_n}$,
by the induction we get $f_M=1_{S_n^{r_n}}*\cdots*1_{S_1^{r_1}}=P_M$. So we can assume the theorem is true for all $L$ such that $M\leq_{deg}L$.

Set $Z_{M'}\in \irr~\Lambda_{V'}$ with $t_i(Z_{M'})=0$ corresponds to $Z_M$. By the construction of Lusztig \cite[the proof of Lemma 2.2]{L4},

\[f_M=1_{S_i^m}*f_{M'}-\sum_{Z_L\in \irr(\Lambda_{V}), t_i(Z_L)>m}\rho_{Z_L}(1_{S_n^m}*f_{M'})f_{L}.\]
By the induction, $f_{M'}=\sum_{M'\leq_{deg}N'}b_{N'}P_{N'}$ with $b_{M'}=1$. From Lemma 3.6 and Lemma 3.7,
we get $1_{S_i^m}*f_{M'}=\sum_{M\leq_{deg}N}c_{N}P_{N}$ with $c_{M}=1$. From Lemma 3.8, we know

\[\sum_{Z_L\in \irr(\Lambda_{V}), t_i(Z_L)>m}\rho_{Z_L}(1_{S_n^m}*f_{M'})f_{L}=\sum_{Z_L\in \irr(\Lambda_{V}),
M\leq_{deg}L, t_i(Z_L)>m}\rho_{Z_L}(1_{S_n^m}*f_{M'})f_{L}.\]

By the second induction hypothesis, $f_{L}$ can be written as the linear combination of $P_{L'}$
where $L\leq_{deg}L'$. Because $M\leq_{deg}L$, we get $f_M=\sum_{N\leq_{deg}M}a_NP_N$ and $a_M=1$.

\end{proof}

We can refine the partial order $\leq_{deg}$ to an total order. It is easy to get the follow Corollary.

\begin{cor}
 Let $\frg=sl_{n+1}(\mathbb{C})$ and $U^+$ be the positive part of the universal enveloping algebra of $\frg$.
 We choose the linear orientation $\overrightarrow{A_n}$ for the Dynkin graph $A_n$ of $\frg$.
 Then the matrix between the PBW bases of type $\overrightarrow{A_n}$ and the semicanonical bases is
 upper triangular unipotent under any order which is compatible with the partial order $\leq_{deg}$.
\end{cor}

\section{Example}

Let $Q$ and its double quiver $\overline{Q}$ as follows:

\[\xymatrix{Q :&1\ar[r]&2},  \ \ \ \ \ \ \UseAllTwocells\xymatrix{\overline{Q}: &1\rtwocell<2>_\alpha^\beta{'}&2}\]

\renewcommand\arraycolsep{0.1cm}
\newcommand{\one}{\begin{array}{c}2\vspace{-10pt}\\1\end{array}}
\newcommand{\two}{\begin{array}{c}1\vspace{-10pt}\\2\end{array}}

Set $V=V_1\oplus V_2$, $\Dim V_1=\Dim V_2=2$. There are three nonisomormphism representations in $\Rep(Q,V)$,
namely, $M_1=1^2\oplus 2^2$, $M_2=1\oplus\two\oplus 2$, $M_3=\two\oplus\two$. We have $M_1\leq_{deg}M_2\leq_{deg}M_3$, so we order them $M_3<M_2<M_1$.

We have

$Z_{M_1}=\overline{\pi^{-1}(\mathcal{O}_{M_1})}=\mathcal{O}_{\scriptsize{\one}\oplus 1
\oplus 2}\cup\mathcal{O}_{1^2\oplus 2^2}\cup\mathcal{O}_{\scriptsize{\one}\oplus \scriptsize{\one}}$

$Z_{M_2}=\overline{\pi^{-1}(\mathcal{O}_{M_2})}=\mathcal{O}_{\scriptsize{\one}\oplus 1\oplus 2}\cup
\mathcal{O}_{\scriptsize{\one}\oplus \scriptsize{\two}}\cup\mathcal{O}_{1^2\oplus 2^2}\cup\mathcal{O}_{\scriptsize{\two}\oplus 1\oplus 2}$

$Z_{M_3}=\overline{\pi^{-1}(\mathcal{O}_{M_3})}=\mathcal{O}_{\scriptsize{\two}\oplus 1\oplus 2}
\cup\mathcal{O}_{\scriptsize{\two}\oplus \scriptsize{\two}}\cup\mathcal{O}_{1^2\oplus 2^2}$

It is easy to calculate the semicanonical basis from the definition:

$f_{M_1}=\frac{1}{4}(1_2^2*1_1^2)$, $f_{M_2}=\frac{1}{2}(1_2*1_1*1_1*1_2)$, $f_{M_3}=\frac{1}{4}(1_1^2*1_2^2)$.

It is easy to see that

$f_{M_1}=P_{M_1}$, $f_{M_2}=P_{M_2}+2P_{M_1}$, $f_{M_3}=P_{M_1}+P_{M_2}+P_{M_3}$.

So, the matrix between $(f_{M_3}, f_{M_2}, f_{M_1})$ and $(P_{M_3}, P_{M_2}, P_{M_1})$ is upper triangular unipotent.

\end{CJK*}

\begin{thebibliography}{100}


\bibitem[BK]{BK} P.Baumann, J.Kamniter, Preprojective algebras and MV polytopes. Represent. Theory, 16(2012), 152-188.

\bibitem[B1]{B1} K.Bongartz, Minimal singularities for representations of Dykin quivers. Comment.Math.Helvetici, 69(1994), 575-611.

\bibitem[B2]{B2} K.Bongartz, On degenerations and extensions of finite dimensional modules. Adv.Math.,  121(1996), 245-287.

\bibitem[GLS]{GLS} C.Gei{\ss}, B.Leclerc, J.Schr\"{o}er, Semicanonical bases and preprojective algebras. Ann.Sci.\'{E}c.Norm.Super, 38(2005), 193-253.

\bibitem[K]{K} M.Kashiwara, On the crystal bases of the q-analogue of universal enveloping algebras. Duke Math.J., 63(1991), 465-516.

\bibitem[KS]{KS} M.Kashiwara, Y.Saito, Geometric construction of crystal bases. Duke Math.J., 89(1997), 9-36.

\bibitem[L1]{L1} G.Lusztig, Canonical bases arising from quantized enveloping algebras. J.Amer.Math.Soc., 3(1990), 447-498.

\bibitem[L2]{L2} G.Lusztig, Canonical bases arising from quantized enveloping algebras \uppercase\expandafter{\romannumeral
2}. Progr.Theoret.Phys.Suppl., 102(1990), 175-201.

\bibitem[L3]{L3} G.Lusztig, Quivers, perverse sheaves, and qunantized enveloping algebras. J.Amer.Math.Soc., 4(1991), 365-421.

\bibitem[L4]{L4} G.Lusztig, Semicanonical bases arising from enveloping algebras. Adv.Math.,  151(2000), 129-139.

\bibitem[S]{S} O.Schiffmann, Lectures on hall algebras.  arXiv:math/0611617v2.



\end{thebibliography}
\end{document}